\documentclass{amsart}

\usepackage{amsmath}
\usepackage{amsfonts}
\usepackage{amssymb}
\usepackage{marvosym}   
\usepackage{amsthm}
\usepackage{thmtools}	

\usepackage{hyperref}   
\usepackage{nameref}    
\usepackage{cleveref}    

\usepackage{breqn}		
\usepackage{lscape}		
\usepackage{graphicx}   
\usepackage{verbatim}   
\usepackage{color}      
\usepackage{subfigure}  

\usepackage{enumerate}  
\usepackage{appendix}   

\newcounter{LemmaCounter}
\newcounter{TheoremCounter}

\newtheorem{theorem}[TheoremCounter]{Theorem}

\newtheorem{corollary}[TheoremCounter]{Corollary}

\newtheorem{definition}[TheoremCounter]{Definition}

\begin{document}

%
%
\title[On a Mean-Value of Gadiyar and Padma]{On a Mean-Value of Gadiyar and Padma: \\ Applying Ramanujan-Fourier Expansions to Conjectures B and D of Hardy and Littlewood}
\author{John Washburn}
\address{N128W12795 Highland Road \\ Germantown, WI 53022}
\urladdr{http://www.WashburnResearch.org}
\email{math@WashburnResearch.org}
\date{\today}

\begin{abstract}

The modified von Mangoldt function,
$\frac{\phi(n)}{n}\Lambda(n)$,
was proved by Hardy \cite{bibitem:Hardy01},
to have a point-wise convergent Ramanujan-Fourier expansion of:
\begin{equation*}
\frac{\phi(n)}{n}\Lambda(n) = 
\sum\limits_{q=1}^{\infty}\frac{\mu(q)}{\phi(q)} c_q(n)
\quad
n \in \mathbb{N}
\end{equation*}
where: $\mu(q)$ is the M\"obius function,
$\phi(q)$ is Euler's totient function,
and $c_q(n)$ is the $q^{th}$ Ramanujan sum defined as:
\begin{equation*}
c_q(n) = \sum\limits_{\substack{k=1 \\ (k,q)=1}}^{q}
e^{2 \pi i \frac{k}{q} n }
\end{equation*}

In their 1999 paper \cite{bibitem:GadiyarPadma:WienerKhinchin},
Gadiyar and Padma imported into sieve theory
the Weiner-Khintchine theorem from the Fourier analysis of stochastic processes,
to produce the heuristic result:
\begin{equation} \label{Eq:Abstract:GadiyarPadmaEquation}
\begin{aligned}
\lim_{N \to \infty}
\frac{1}{N}
\sum\limits_{n=1}^{N}
\frac{\phi(n)}{n}\Lambda(n)
\frac{\phi(n+h)}{n+h}\Lambda(n+h)
&=
\lim_{N \to \infty}
\frac{1}{N}
\sum\limits_{n=1}^{N}
\left[
\sum\limits_{q_1=1}^{\infty}\frac{\mu(q_1)}{\phi(q_1)} c_{q_1}(n)
\sum\limits_{q_2=1}^{\infty}\frac{\mu(q_2)}{\phi(q_2)} c_{q_2}(n+h)
\right]
\\
&=
\sum\limits_{q=1}^{\infty}
\left\Vert \frac{\mu(q)}{\phi(q)} \right\Vert^2
c_q(h)
\end{aligned}
\end{equation}
for any fixed and finite $h$.
If the interchange of limits,
$ 
\left(
\text{first:}
\lim\limits_{N \to \infty}
\text{followed by:}
\lim\limits_{q_1,q_2 \to \infty}
\right)
$ 
is justified, then the results of \eqref{Eq:Abstract:GadiyarPadmaEquation} follow easily.
The main result of this paper is to prove
the Gadiyar-Padma equation,
\eqref{Eq:Abstract:GadiyarPadmaEquation},
is true by finessing the interchange of limits using Abel summation.
\\

From the proof of \eqref{Eq:Abstract:GadiyarPadmaEquation},
the quantitative form of the Hardy-Littlewood prime pair conjecture:
\begin{equation} \label{Eq:Abstract:HardyLittlewoodPPC}
\begin{aligned}
\lim_{N \to \infty}
\frac{1}{N}
\sum\limits_{n=1}^{N}
\left( \frac{\phi(n  ) \Lambda(n  )}{n  } \right)
\left( \frac{\phi(n+h) \Lambda(n+h)}{n+h} \right)
&= C_{h}
\\
&= \begin{cases}
2
\prod\limits_{\substack{p \mid h \\ p > 2}} \left( \frac{p-1}{p-2} \right)
\prod\limits_{p>2} \left( 1 - \frac{1}{\left( p - 1 \right)^2} \right)
& \text{h is even}
\\
0 & \text{h is odd}
\end{cases}
\end{aligned}
\end{equation}
follows and is also true.
\end{abstract}

\maketitle

\newpage
\section{Introduction}

One of the remarkable achievements of
Ramanujan\cite{bibitem:Ramanujan01}, Hardy\cite{bibitem:Hardy01} and Carmichael\cite{bibitem:Carmichael01}
was the development of Ramanujan-Fourier series which converge (point-wise) to an arithmetic function.
The Ramanujan-Fourier expansion, for an arithmetic function, $a(n)$, can be given by:
\begin{equation*}
a(n) = \sum\limits_{q=1}^{\infty} a_q c_q(n)
\end{equation*}
where the $q^{th}$ Ramanujan sum, $c_q(n)$, is defined as
\begin{equation*}
c_q(n) = \sum\limits^{q}_{\substack{k=1 \\ (k,q)=1}} e^{2 \pi i \frac{k}{q} n }
\end{equation*}
and $(k,q)$ is the greatest common divisor of $k$ and $q$.
The Ramanujan-Fourier coefficient, $a_q$, is given by:
\begin{equation*}
a_q = \frac{1}{\phi(q)} \lim_{N \to \infty} \sum_{n=1}^{N} a(n) c_q(n)
\end{equation*}
where $\phi(q)$ is the Euler totient function.
For a detailed survey of Ramanujan sums and Ramanujan-Fourier expansions, the author
refers the reader to the texts:
\cite{bibitem:Sivaramakrishnan:01},
\cite{bibitem:SchwarzSpilker:01},
\cite{bibitem:Lucht:01},  and
\cite{bibitem:Murty:01}.

H. G. Gadiyar and R. Padma in their 1999 paper \cite{bibitem:GadiyarPadma:WienerKhinchin}
imported to number theory the Wiener-Khintchine theorem from signal analysis.
The Wiener-Khintchine theorem is used in the Fourier analysis of chaotic or stochastic processes,
such as Brownian motion, signal noise, and signals embedded within a significant amount of noise.
The 1999 paper \cite{bibitem:GadiyarPadma:WienerKhinchin} paper applied the Wiener-Khintchine theorem to the
Ramanujan-Fourier expansions of the arithmetic functions used in various sieve methods.

The central relation to be proved in
\cite{bibitem:GadiyarPadma:WienerKhinchin},
\cite{bibitem:GadiyarPadma:SophieGermainPrimes}, and
\cite{bibitem:GadiyarPadma:HardyLittleWoodConjecture:D},
is the Gadiyar-Padma equation:
\begin{equation} \label{eq:GadiyarPadmaEquation}
\begin{aligned}
\lim_{N \to \infty}
\frac{1}{N}
\sum\limits_{n=1}^{N}
\frac{\phi(n) \Lambda(n)}{n}
&\frac{\phi(n+h) \Lambda(n+h)}{n+h}
\\
&=
\lim_{N \to \infty}
\frac{1}{N}
\sum\limits_{n=1}^{N}
\sum\limits_{q_1 = 1}^{\infty}
\frac{\mu(q_1)}{\phi(q_2)} c_{q_1}(n)
\sum\limits_{q_2 = 1}^{\infty}
\frac{\mu(q_2)}{\phi(q_2)} c_{q_2}(n+h)
\\
&\stackrel{?}{=}
\sum\limits_{q_1 = 1}^{\infty}
\sum\limits_{q_2 = 1}^{\infty}
\frac{\mu(q_1)}{\phi(q_2)}
\frac{\mu(q_2)}{\phi(q_2)}
\left[
\lim_{N \to \infty}
\frac{1}{N}
\sum\limits_{n=1}^{N}
c_{q_1}(n)
c_{q_2}(n+h)
\right]
\\
&=
\sum\limits_{q = 1}^{\infty}
\left\Vert \frac{\mu(q)}{\phi(q)}\right\Vert^2
c_{q}(h)
\end{aligned}
\end{equation}
Which is the Wiener-Khintchine theorem applied to the auto-correlation
of the arithmetic function, $\frac{\phi(n) \Lambda(n)}{n}$.
The Gadiyar-Padma equation uses the Ramanujan-Fourier expansion of an arithmetic function to formulate
the average density of those primes and prime powers which survive a sieve.

The study of Ramanujan sums and the study of Ramanujan-Fourier expansions has progressed significantly since 1918.
Among the discoveries are:
\begin{itemize}
\item New methods for determining the R-F coefficients \cite{bibitem:Murty:02},
\item Analysis of  the convolution of functions from different measure spaces \cite{bibitem:Spilker:01}, and
\item The mean value of sums of products of Ramanujan sums \cite{bibitem:Toth:01}
\end{itemize}
The difficulty presented by the Gadiyar-Padma equation though is that the Ramanujan-Fourier expansion under study is not
tractable enough or well behaved enough to justify the interchange of limits proposed in
\cite{bibitem:GadiyarPadma:WienerKhinchin},
\cite{bibitem:GadiyarPadma:SophieGermainPrimes}, or
\cite{bibitem:GadiyarPadma:HardyLittleWoodConjecture:D},
\textbf{and} still maintain convergence. This is because the expansion,
$\frac{\phi(n) \Lambda(n)}{n} = \sum\limits_{q = 1}^{\infty} \frac{\mu(q)}{\phi(q)} c_{q}(n)$,
in \eqref{eq:GadiyarPadmaEquation}
\begin{itemize}
\item is only conditionally convergent and not absolutely convergent,
\item is only point-wise convergent in $n$ and not uniformly convergent in $n$,
\item is not an almost even function, $\mathcal{B}^q$, for some $1 \le q$, as discussed in \cite{bibitem:SchwarzSpilker:01}
\item is not a member of, $\mathcal{C}^q$, as discussed in \cite{bibitem:Lucht:01}
\item has R-F coefficients, $a_q$, such that the expression, $a_q \phi(q)$, is not square-summable.
This removes the expansion from the space of almost periodic functions belonging to the Besicovitch space, $B^2$, as discussed in:
\cite{bibitem:Besicovitch:01} and \cite{bibitem:ParsevalEquation_by_MPeter}.

\end{itemize}
If the arithmetic function under study were within one of these function spaces,
then the application of the Wiener-Khintchine theorem to the Ramanujan-Fourier expansions
would be easily justified.  This is analogous to the treatment of the Wiener-Khintchine theorem within most texts on signal analysis which begin with
\textit{\textbf{If a function has a Fourier transforms of: $\cdots$}}.
But, if a function possesses a Fourier transform,
then the proof of the Wiener-Khintchine theorem is trivial.
The power and utility of the Wiener-Khintchine theorem is that it also applies
to signals which do not possess a Fourier transform. As discussed by Wu and Vaswani
in \cite{bibitem:LuVaswani:01} and \cite{bibitem:LuVaswani:02},
the interchange of limits must be justified in a signal processing context as well.
Wu and  Vaswani note that within several prominent texts on signal processing
purporting``prove'' the Wiener-Khintchine theorem, the justification  for interchanging the limits is left unstated
or the stated justification is not sufficient.

This paper presents the Ramanujan-Abel method which permits the interchange of limits under far more relaxed conditions,
but does so at cost of temporarily accepting summability in lieu of convergence.
The limits can be interchanged, but the resulting expression is then
merely summable to some value rather than convergent to that value.
Abel's theorem for a power series is then used to demonstrate the desired convergence exists.

\newpage
\section{Definitions and Notation}

This paper has the following notational conventions:
\begin{itemize}
\item $n$ is a positive integer.
\item $x$ is real.
\item $p$ is an arbitrary prime.
\item $q$ is a positive integer unless the context clearly indicates otherwise.
\item $z$ is a real within the open interval: $0 < z < 1$.
\end{itemize}

\begin{definition} \label{def:TwinPrimeConstant}
Throughout this paper $C_2$ is the the twinned primes constant from \cite{bibitem:HardyLittlewood01} and is given by:
\begin{equation*}
C_2 = \prod\limits_{p > 2} \left( 1 - \frac{1}{\left( p - 1\right)^2} \right)
\end{equation*}
The approximate value of $C_2$ is:
\begin{equation*}
C_2 = 0.6601618158\dots
\end{equation*}
\end{definition}

\begin{definition} \label{def:RamanujanSum:Real}
The real-value Ramanujan sum, $c_q\left( x \right)$, is defined as:
\begin{equation*}
c_q\left( x \right) =
\begin{cases}
1 & q = 0
\\
\cos\left( 2 \pi x \right) & q = 1
\\
\cos\left( \pi x \right) & q = 2
\\
2 \sum\limits_{\substack{k=1 \\ (k,q)=1}}^{\left\lfloor \frac{q}{2} \right\rfloor}
\cos\left( 2 \pi \frac{k}{q} x \right) & q \ge 3
\end{cases}
\end{equation*}
\end{definition}

\begin{definition} \label{def:ModifiedLambda}
The modified von Mangoldt function, $\Lambda_{1}\left( n \right)$, is defined as:
\begin{equation*}
\begin{aligned}
\Lambda_{1}\left( n \right)
&=
\frac{\phi(n) }{ n } \Lambda(n)
\\
&=
\begin{cases}
\left( \frac{ p-1 }{p} \right) \log{p}  & n = p^k \text{, } 1 \le k \text{, and $p$ is prime}
\\
0 & \text{ otherwise}
\end{cases}
\end{aligned}
\end{equation*}
where $\Lambda(n)$ is the von Mangoldt function defined as:
\begin{equation*}
\Lambda(n) =
\begin{cases}
\log{p} & n = p^k \text{ } 1 \le k \text{ and $p$ is prime}
\\
0 & \text{ otherwise}
\end{cases}
\end{equation*}
\end{definition}

\begin{definition} \label{def:ModifiedLambda:Reals}
The modified von Mangoldt function can be extended to the reals with the following definition:
\begin{equation*}
\Lambda_1\left( x \right)
\stackrel{def}{=}
\sum\limits_{q=1}^{\infty}
\frac{\mu\left( q \right)}{\phi\left( q \right)}
c_q\left( x \right)
\end{equation*}
\end{definition}

\begin{definition} \label{def:BunyakovskyCondition:ChebyshevLikeFunction}
A Bunyakovsky set is a set of $m$ functions,
$f_1(T), f_2(T), f_3(T), \cdots, f_{m-1}(T), f_{m}(T)$, for $m \ge 1$ and $f_j(T) \in \mathbb{Z}\left[ T \right]$
\\
where:
\begin{enumerate}[(a)]
\item Each $f_j$ has a positive leading coefficient.
\item Each of the $f_j$'s is irreducible in $\mathbb{Q}\left[ T \right]$
\item The $f_j$'s are pair-wise prime in $\mathbb{Q}\left[ T \right]$
\item For no prime, $p$, is the product, $f(n) = f_1(n) \cdots f_m(n)$, divisible by $p$ for all integers, $n$.
That is, the function, $f:\mathbb{Z} \to \mathbb{Z}/(p)$, is not identically zero for any prime $p$.
\end{enumerate}

See
\cite{bibitem:Baier:Bateman-Horn},
\cite{bibitem:Bateman-Horn01}, and
\cite{bibitem:Conrad:HardyLittlewoodConstants}
for more information on the the Bateman-Horn Conjecture and its relation to the Bunyakovsky conditions.
\end{definition}

\begin{definition} \label{def:RamanujanExpansion:PowerSeries}
Given an arithmetic function, $f(n)$, with a point-wise convergent Ramanujan-Fourier (R-F) expansion of:
\begin{equation*}
f(n) = \sum\limits_{q=1}^{\infty} a_q c_q(n)
\end{equation*}
then the related functions, $f(z, n)$ and $f(z, x)$,
are the Ramanujan expansion power series and are given by:
\begin{equation*}
\begin{aligned}
f(z, n)
&\stackrel{def}{=}
\sum\limits_{q=1}^{\infty} a_q z^q c_q(n)
\\
f(z, x)
&\stackrel{def}{=}
\sum\limits_{q=1}^{\infty} a_q z^q c_q(x)
\end{aligned}
\end{equation*}
\end{definition}

\begin{definition} \label{def:RamanujanExpansion:PartialSummation}
Given a Ramanujan expansion power series of:
\begin{align*}
f(z, n) &= \sum\limits_{q=1}^{\infty} a_q z^q c_q(n)
\\
&\text{or}
\\
f(z, x) &= \sum\limits_{q=1}^{\infty} a_q z^q c_q(x)
\end{align*}
then the partial summation of the Ramanujan expansion power series for the first $Q$ terms is given by:
\begin{align*}
f(Q, z, n) 
&\stackrel{def}{=}
\sum\limits_{q=1}^{Q} a_q z^q c_q(n)
\\
f(Q, z, x) 
&\stackrel{def}{=}
\sum\limits_{q=1}^{Q} a_q z^q c_q(x)
\end{align*}
\end{definition}

\newpage
\section{The Convergence Properties of the Ramanujan expansion power series related to  $\Lambda_1\left( n\right)$ }

\begin{theorem} \label{thm:AbsoluteConvergence:RFPowerSeries:Lambda1}
For any fixed $z$ on the open interval, $0 < z < 1$, the Ramanujan expansion power series,
$\Lambda_1\left( z, x \right)$,
converges absolutely for all $x \in \mathbb{R}$
\end{theorem}

\begin{proof}
\begin{equation} \label{eq:Lambda1:RealValuedVersion}
\Lambda_1\left( z, x \right)
=
\sum\limits_{q=1}^{\infty}
\frac{\mu\left( q \right)}{\phi\left( q \right)}
c_q\left( x \right)
z^q
\end{equation}

On the open interval, $0 < z < 1$,
the right hand sided of \eqref{eq:Lambda1:RealValuedVersion} converges absolutely
because:
\begin{align*}
\sum\limits_{q=1}^{\infty}
\frac{\mu\left( q \right)}{\phi\left( q \right)}
c_q\left( x \right)
z^q
&\le
\sum\limits_{q=1}^{\infty}
\left\vert
\frac{\mu\left( q \right)}{\phi\left( q \right)}
c_q\left( x \right)
z^q
\right\vert
\\
&\le
\sum\limits_{q=1}^{\infty}
\frac{\left\vert \mu\left( q \right) \right\vert }{\phi\left( q \right)}
\phi\left( q \right)
z^q
\\
&\le
\sum\limits_{q=1}^{\infty}
z^q
\\
&\le
\frac{z}{1- z}
\end{align*}
Since the convergence of
\begin{equation*}
\sum\limits_{q=1}^{\infty}
\left\vert
\frac{\mu\left( q \right)}{\phi\left( q \right)}
c_q\left( x \right)
z^q
\right\vert
\end{equation*}
depends on $z$, but is without regard to $x$, the convergence of the power series,
$\Lambda_1\left(z , x \right)$,
is absolute for all $x \in \mathbb{R}$.
\end{proof}

\begin{theorem} \label{thm:UniformConvergence:RFPowerSeries:Lambda1}
For any fixed, $z$, on the open interval, $0 < z < 1$, sequence of functions,
$\Lambda_1\left(Q, z, x \right)$,
converge uniformly to the Ramanujan expansion power series,
$\Lambda_1\left( z, x \right)$,
for all $x \in \mathbb{R}$.
\end{theorem}

\begin{proof}

\begin{equation} \label{eq:Lambda1:Tails}
\begin{aligned}
\left\vert \Lambda_1\left(z, x \right) - \Lambda_1\left(Q, z, x \right) \right\vert
&=
\left\vert 
\sum\limits_{q=1}^{\infty}
\frac{\mu\left( q \right)}{\phi\left( q \right)}
c_q\left( x \right)
z^q
-
\sum\limits_{q=1}^{Q}
\frac{\mu\left( q \right)}{\phi\left( q \right)}
c_q\left( x \right)
z^q
\right\vert
\\
&=
\left\vert 
\sum\limits_{q=Q+1}^{\infty}
\frac{\mu\left( q \right)}{\phi\left( q \right)}
c_q\left( x \right)
z^q
\right\vert
\end{aligned}
\end{equation}

For any fixed $z$ within the open interval, $0 < z < 1$,
the right hand sided of \eqref{eq:Lambda1:Tails} is bounded by:
\begin{equation} \label{eq:BoundedTails}
\begin{aligned}
\left\vert 
\sum\limits_{q=Q+1}^{\infty}
\frac{\mu\left( q \right)}{\phi\left( q \right)}
c_q\left( x \right)
z^q
\right\vert
&\le
\sum\limits_{q=Q+1}^{\infty}
\left\vert 
\frac{\mu\left( q \right)}{\phi\left( q \right)}
c_q\left( x \right)
z^q
\right\vert
\\
&\le
\sum\limits_{q=Q+1}^{\infty}
\left\vert 
\frac{\mu\left( q \right)}{\phi\left( q \right)}
\right\vert
\left\vert 
c_q\left( x \right)
\right\vert
\left\vert 
z^q
\right\vert
\\
&\le
\sum\limits_{q=Q+1}^{\infty}
\frac{1}{\phi\left( q \right)}
\phi\left( q \right)
z^q
\\
&\le
\sum\limits_{q=Q+1}^{\infty}
z^q
\\
&\le
\frac{z^{Q+1}}{1-z}
\end{aligned}
\end{equation}

Combining \eqref{eq:Lambda1:Tails} and \eqref{eq:BoundedTails} simplifies to:
\begin{equation} \label{eq:PartialSummation:Tail:Bounds}
\left\vert \Lambda_1\left(z, x \right) - \Lambda_1\left(Q, z, x \right) \right\vert
\le
z^Q
\left( \frac{z}{1-z} \right)
\end{equation}

Since, $0 < z < 1$, the right hand side of the \eqref{eq:PartialSummation:Tail:Bounds}
forms a strictly decreasing sequence in $Q$ such that:
\begin{equation*}
z^{Q+N}
\left( \frac{z}{1-z} \right)
<
z^{Q+N-1}
\left( \frac{z}{1-z} \right)
\cdots
<
z^{Q+2}
\left( \frac{z}{1-z} \right)
<
z^{Q+1}
\left( \frac{z}{1-z} \right)
<
z^Q
\left( \frac{z}{1-z} \right)
\end{equation*}
for all integers $N \ge 0$.

Thus, for any fixed $z$ in the interval: $0 < z < 1$, and for any $0 < \epsilon$, it is possible to select a value, $Q_z\left( \epsilon \right)$,
such that
\begin{equation*}
z^{Q} \left( \frac{z}{1-z} \right) < \epsilon
\end{equation*}
for all $Q \ge Q_z\left( \epsilon \right)$.

Since the selection of $Q_z\left( \epsilon \right)$ depends only on $z$ and does not depend on $x$,
the sequence of functions, $\Lambda_1\left(Q, z, x \right)$,
converges uniformly in $x \in \mathbb{R}$ 
to the Ramanujan expansion power series, $\Lambda_1\left(z, x \right)$
for every fixed $z$ on the interval $0 < z < 1$.

\end{proof}

\begin{corollary} \label{cor:AbsoluteConvergence:Lambda1:Integers}
A corollary to \Cref{thm:AbsoluteConvergence:RFPowerSeries:Lambda1} is 
on the open interval, $0 < z < 1$, the Ramanujan expansion power series,
$\Lambda_1\left(z,  n \right)$,
converges absolutely for all $n \in \mathbb{Z}$
\end{corollary}

\begin{corollary} \label{cor:UniformConvergence:Lambda1:Integers}
A corollary to \Cref{thm:UniformConvergence:RFPowerSeries:Lambda1} is 
for any fixed $z$ within the open interval, $0 < z < 1$, the sequence of functions,
$\Lambda_1\left(Q, z,  n \right)$,
converges uniformly to $\Lambda_1\left(z,  n \right)$
for all $n \in \mathbb{Z}$
\end{corollary}

\begin{theorem} \label{thm:Convergence:RFPowerSeries:Lambda1}
In the limit, the Ramanujan expansion power series related to
$\Lambda_1(z, n)$ converges point-wise to:
\begin{equation*}
\lim\limits_{z \to 1^-}
\left[
\sum\limits_{q=1}^{\infty}
\frac{\mu\left( q \right)}{\phi\left( q \right)}
c_q\left( n \right)
z^q
\right]
=
\Lambda_1\left( n \right)
\end{equation*}
for $n \in \mathbb{N}$
\end{theorem}

\begin{proof}
The Ramanujan-Fourier (R-F) expansion of $\Lambda_1\left( n \right)$ is given by:
\begin{equation} \label{eq:Lambda1:InitalRFExpansion}
\Lambda_1\left( n \right)
=
\sum\limits_{q=1}^{\infty}
\frac{\mu\left( q \right)}{\phi\left( q \right)}
c_q\left( n \right)
\quad
n \in \mathbb{N}
\end{equation}
The point-wise convergence of \eqref{eq:Lambda1:InitalRFExpansion} was proved by Hardy in \cite{bibitem:Hardy01}.
By \Cref{thm:AbsoluteConvergence:RFPowerSeries:Lambda1} the power series:
\begin{equation*}
\sum\limits_{q=1}^{\infty}
\frac{\mu\left( q \right)}{\phi\left( q \right)}
c_q\left( x \right)
z^q
\end{equation*}
is convergent for all $x \in \mathbb{R}$ for any fixed $z$ within the open interval, $0 < z < 1$.
Abel's theorem states that, in the limit, the power series converges point-wise to:
\begin{equation} \label{eq:PowerSeries:RFExpansion}
\lim\limits_{r \to 1^-}
\left[
\sum\limits_{q=1}^{\infty}
\frac{\mu\left( q \right)}{\phi\left( q \right)}
c_q\left( x \right)
z^q
\right]
=
\sum\limits_{q=1}^{\infty}
\frac{\mu\left( q \right)}{\phi\left( q \right)}
c_q\left( x \right)
\quad
x \in \mathbb{R}
\end{equation}
Combining \eqref{eq:Lambda1:InitalRFExpansion} and \eqref{eq:PowerSeries:RFExpansion} yields:
\begin{equation} \label{eq:Lambda1:RFExpansion:PowerSeries}
\lim\limits_{z \to 1^-}
\left[
\sum\limits_{q=1}^{\infty}
\frac{\mu\left( q \right)}{\phi\left( q \right)}
c_q\left( n \right)
z^q
\right]
=
\Lambda_1\left( n \right)
\quad
n \in \mathbb{N}
\end{equation}

\end{proof}

\newpage
\section{Proof of Conjecture B of Hardy and Littlewood}

In \cite{bibitem:GadiyarPadma:WienerKhinchin} Gadiyar and Padma tackle the
question of the prime pairs conjecture of which the twin prime conjecture is a special case.
The prime pairs conjecture is Conjecture B from \cite{bibitem:HardyLittlewood01}.
In \cite{bibitem:GadiyarPadma:WienerKhinchin} (absent a justification for the interchange of limits) the following is ``proved'':
\begin{equation} \label{eq:GadiyarPadma:PPCResult}
\lim_{N \to \infty}
\frac{1}{N}
\sum\limits_{n = 1}^{N}
\frac{\phi\left( n \right) \Lambda\left( n \right)}{n}
\frac{\phi\left( n + 2 h \right) \Lambda\left( n + 2 h \right)}{n + 2 h}
=
2 C_2
\prod\limits_{\substack{p > 2 \\ p \mid 2 h}} \left( \frac{p-1}{p-2}\right)
\end{equation}
\Cref{eq:GadiyarPadma:PPCResult} (if true) is sufficient to prove the existential portion of
Conjecture B;
\textit{
For every positive, even number $2 h$, there exist an infinite number of prime pairs $p$, $p + 2h$.
}
The Ramanujan-Abel method can be used to prove \eqref{eq:GadiyarPadma:PPCResult}
is true for any fixed, even number, $2 h$.

\begin{theorem} [Prime Pairs are Infinite -- Conjecture B of \cite{bibitem:HardyLittlewood01}] \label{thm:PrimePairsAreInfinite}
For every positive, even number, $k = 2 h$, there are infinitely many prime pairs, $p,p\prime = p + k$.
\end{theorem}

\begin{proof}
By \cref{cor:UniformConvergence:Lambda1:Integers}, 
the convergence of sequence of partial summations 
of the Ramanujan expansion power series
in \eqref{eq:PPC:PowerSeriesSieve} is uniform in $n$,
so the limits may be interchanged to yield:
\begin{equation} \label{eq:PPC:PowerSeriesSieve}
\begin{aligned}
\lim_{N \to \infty}
\lim_{Q \to \infty}
\frac{1}{N}
\sum\limits_{n = 1}^{N}
&\sum\limits_{{q_1}=1}^{Q}
\frac{\mu({q_1})}{\phi({q_1})}
c_{q_1}\left( n \right)
z^{q_1}
\sum\limits_{{q_2}=1}^{Q}
\frac{\mu({q_2})}{\phi({q_2})}
c_{q_2}\left( n + 2 h \right)
z^{q_2}
\\
&=
\lim_{Q \to \infty}
\sum\limits_{{q_1}=1}^{Q}
\sum\limits_{{q_2}=1}^{Q}
\frac{\mu({q_1})}{\phi({q_1})}
\frac{\mu({q_2})}{\phi({q_2})}
z^{\left( q_2 + q_1 \right)}
\left[
\lim_{N \to \infty}
\frac{1}{N}
\sum\limits_{n = 1}^{N}
c_{q_1}\left( n \right)
c_{q_2}\left( n + 2 h \right)
\right]
\end{aligned}
\end{equation}
As listed in \Cref{appendix:Properties:RamanujanSums}:
\begin{equation*}
\lim_{N \to \infty}
\frac{1}{N}
\sum\limits_{n = 1}^{N}
c_{q_1}\left( n \right)
c_{q_2}\left( n + 2 h \right)
=
\begin{cases}
c_{q}\left( 2 h \right) & q = q_1 = q_2
\\
0 & \text{ otherwise}
\end{cases}
\end{equation*}
Applying this mean value to \eqref{eq:PPC:PowerSeriesSieve} yields:
\begin{equation} \label{eq:Theorem8:PowerSeriesConvergence}
\lim_{N \to \infty}
\lim_{Q \to \infty}
\frac{1}{N}
\sum\limits_{n = 1}^{N}
\sum\limits_{{q_1}=1}^{Q}
\frac{\mu({q_1})}{\phi({q_1})}
c_{q_1}\left( n \right)
z^{q_1}
\sum\limits_{{q_2}=1}^{Q}
\frac{\mu({q_2})}{\phi({q_2})}
c_{q_2}\left( n + 2 h \right)
z^{q_2}
=
\sum\limits_{q=1}^{\infty}
\frac{\mu(q)}{\phi(q)}
c_q\left( 2 h \right)
z^{2 q }
\end{equation}
\Cref{eq:Theorem8:PowerSeriesConvergence} is true for all fixed $z$ within the open interval: $0 < z < 1$.
Taking the limit, $z \to 1^-$, and applying \cref{thm:Convergence:RFPowerSeries:Lambda1} yields:
\begin{equation} \label{eq:PPC:Penultimate}
\lim_{N \to \infty}
\frac{1}{N}
\sum\limits_{n = 1}^{N}
\Lambda_1\left( n \right)
\Lambda_1\left( n + 2 h\right)
=
\sum\limits_{q=1}^{\infty}
\frac{\mu(q)}{\phi(q)}
c_q\left( 2 h \right)
\end{equation}

Gayidar and Padma prove in \cite{bibitem:GadiyarPadma:WienerKhinchin} that
\begin{equation*}
\sum\limits_{q=1}^{\infty}
\frac{\mu(q)}{\phi(q)}
c_q\left( 2 h \right)
=
2 C_2
\prod\limits_{\substack{p > 2 \\ p \mid 2 h}}
\left( \frac{p-1}{p-2}\right)
\end{equation*}
Which is the Hardy-Littlewood constant for Conjecture B.

Combining this result with \eqref{eq:PPC:Penultimate} leads to:
\begin{equation*} \label{eq:PPC:HardyLittlewoodConstant}
\lim_{N \to \infty}
\frac{1}{N}
\sum\limits_{n = 1}^{N}
\Lambda_1\left( n \right)
\Lambda_1\left( n + 2 h\right)
=
2 C_2
\prod\limits_{\substack{p > 2 \\ p \mid 2 h}}
\left( \frac{p-1}{p-2}\right)
\end{equation*}
Since
\begin{equation*}
\frac{\phi\left( n \right) }{n} \le 1
\quad
\forall \, n \ge 1
\end{equation*}
the classic sieve for prime pairs has the lower limit given by:
\begin{equation} \label{eq:PPC:ClassicSieve}
\begin{aligned}
0 &<
2 C_2
\prod\limits_{\substack{p > 2 \\ p \mid 2 h}}
\left( \frac{p-1}{p-2}\right)
\\
&=
\lim_{N \to \infty}
\frac{1}{N}
\sum\limits_{n = 1}^{N}
\Lambda_1\left( n \right)
\Lambda_1\left( n + 2 h\right)
\\
\\
&=
\lim_{N \to \infty}
\frac{1}{N}
\sum\limits_{n = 1}^{N}
\left[
\frac{
\phi\left( n \right) 
}{n}
\Lambda\left( n \right)
\right]
\left[
\frac{
\phi\left( n + 2 h \right) 
}{n + 2 h }
\Lambda\left( n + 2 h\right)
\right]
\\
&\le
\lim_{N \to \infty}
\frac{1}{N}
\sum\limits_{n = 1}^{N}
\Lambda\left( n \right)
\Lambda\left( n + 2 h\right)
\end{aligned}
\end{equation}
Thus, \eqref{eq:PPC:ClassicSieve}, simplifies to:
\begin{equation} \label{eq:PPC:ClassicSieve:Final}
0 < 2 C_2
\prod\limits_{\substack{p > 2 \\ p \mid 2 h}}
\left( \frac{p-1}{p-2}\right)
\le
\lim_{N \to \infty}
\frac{1}{N}
\sum\limits_{n = 1}^{N}
\Lambda\left( n \right)
\Lambda\left( n + 2 h\right)
\end{equation}

Because the mean value of the sieve in \eqref{eq:PPC:ClassicSieve:Final} is greater than zero,
\eqref{eq:PPC:ClassicSieve:Final} suffices to demonstrate the number of prime pairs separated by a gap of $2 h$ is infinite.
\end{proof}

\newpage
\section{Proof of Conjecture D of Hardy and Littlewood}

In \cite{bibitem:GadiyarPadma:SophieGermainPrimes}
Gadiyar and Padma tackle the density of Sophie Germain Primes.
In \cite{bibitem:GadiyarPadma:HardyLittleWoodConjecture:D}
Gadiyar and Padma consider Conjecture D of \cite{bibitem:HardyLittlewood01}.
Sophie Germain Primes are prime pairs, $p$ and $p^{\prime}$, such that $p^{\prime} = 2 p + 1$.
Conjucture D from \cite{bibitem:HardyLittlewood01} states that under suitable conditions for $a$, $b$, and $\ell$
there exists an infinite number of  prime pairs, $p$ and $p^{\prime}$,
which are the solution to: $a p^{\prime} - b p = \ell$.
Sophie Germain Primes are a special case of Conjecture D where: $a = 1$, $b=2$, and $\ell=1$.
Twinned primes are a special case of Conjecture D where: $a = 1$, $b=1$, and $\ell=2$.

\begin{theorem} [Conjecture D of \cite{bibitem:HardyLittlewood01}] \label{thm:SophieGermainPrimeAreInfinite}
If
\begin{itemize}
\item $a$, $b$ and $\ell$ are distinct, positive integers,
\item $a$, $b$ and $\ell$ are pair-wise relatively prime; i.e. $(a, b) = 1$, $(a, \ell) = 1$, $(\ell, b) = 1$,
\item one and only one of $a$, $b$, $\ell$ is even, and
\item $P(n)$ is the number of prime pairs, $p$ and $p\prime$,  which are a solution to: $ap - bp\prime = \ell$,
\end{itemize}
Then, there are an infinite number of primes pairs which are the solution to: $ap - bp\prime = \ell$ and
\begin{equation*}
\lim\limits_{n \to \infty}
\frac{ P\left( n \right) }{n}
=
\frac{2 C_2}{a}
\prod\limits_{p} \left( \frac{p-1}{p-2}\right)
\end{equation*}
where the product extends over all odd primes $p$ which divide $a$, $b$ or $\ell$.
\end{theorem}

\begin{proof}

The proof here follows the heuristic argument found in
\cite{bibitem:GadiyarPadma:SophieGermainPrimes} and
\cite{bibitem:GadiyarPadma:HardyLittleWoodConjecture:D},
but begins with the Ramanujan expansion power series,
\begin{equation*}
\Lambda_1\left( z, x \right) 
= \sum\limits_{q=1}^{\infty}
\frac{\mu\left( q \right)}{\phi\left( q \right)}
c_{q} \left( x \right)
z^{q}
\end{equation*}
rather than the R-F expansions,
\begin{equation*}
\Lambda_1\left( n \right) 
= \sum\limits_{q=1}^{\infty}
\frac{\mu\left( q \right)}{\phi\left( q \right)}
c_{q} \left( n \right)
\end{equation*}
used in 
\cite{bibitem:GadiyarPadma:SophieGermainPrimes} and
\cite{bibitem:GadiyarPadma:HardyLittleWoodConjecture:D}.

Begining with:
\begin{equation} \label{eq:ConjectureD:StartingPoint}
\begin{aligned}
\lim_{N \to \infty}
&\frac{1}{N}
\sum\limits_{\substack{n = 1 \\ a \mid \left( b n + \ell \right)}}^{N}
\Lambda_1\left( z, n \right)
\Lambda_1\left( z, \frac{b n + \ell}{a} \right)
\\
&=
\lim_{N \to \infty}
\frac{1}{N}
\sum\limits_{n = 1}^{N}
\left(
\sum\limits_{q_1=1}^{\infty}
\frac{\mu\left( q_1 \right)}{\phi\left( q_1 \right)}
c_{q_1} \left( n \right)
z^{q_1}
\right)
\left(
\sum\limits_{q_2=1}^{\infty}
\frac{\mu\left( q_2 \right)}{\phi\left( q_2 \right)}
c_{q_2} \left( \frac{b n + \ell}{a} \right)
z^{q_2}
\right)
\left(
\frac{1}{a}
\sum\limits_{j = 0}^{a-1}
e^{2 \pi i \left( \frac{b n + \ell }{a} \right)  j }
\right)
\end{aligned}
\end{equation}
where:
$c_{q_2} \left( n \right)$ and
$c_{q_2} \left( \frac{b n + \ell}{a} \right)$
are the real-valued version of the Ramanujan sums
and the last term:
\begin{equation*}
\frac{1}{a}
\sum\limits_{j = 0}^{a-1}
e^{2 \pi i \left( \frac{b n + \ell }{a} \right)   j }
=
\begin{cases}
1 & a \mid \left( b n + \ell \right)
\\
0 & \text{otherwise}
\end{cases}
\end{equation*}
limits the arguments of the summands to integer values.

By \Cref{thm:UniformConvergence:RFPowerSeries:Lambda1} the 
sequence of partial summations of the Ramanujan expansion power series 
converges unifomly in $x$,
so the limits found on the RHS of \Cref{eq:ConjectureD:StartingPoint},
$N \to \infty$ and $Q \to \infty$,
may be interchanged to yield:
\begin{equation} \label{eq:Theorem9:InterchangedLimits}
\begin{aligned}
\lim_{N \to \infty}
\frac{1}{N}
\sum\limits_{n = 1}^{N}
&\left(
\sum\limits_{q_1=1}^{\infty}
\frac{\mu\left( q_1 \right)}{\phi\left( q_1 \right)}
c_{q_1} \left( n \right)
z^{q_1}
\right)
\left(
\sum\limits_{q_2=1}^{\infty}
\frac{\mu\left( q_2 \right)}{\phi\left( q_2 \right)}
c_{q_2} \left( \frac{b n + \ell}{a} \right)
z^{q_2}
\right)
\left(
\frac{1}{a}
\sum\limits_{j = 0}^{a-1}
e^{2 \pi i \left( \frac{b n + \ell }{a} \right)  j }
\right)
\\
&=
\frac{1}{a}
\sum\limits_{j = 0}^{a-1}
\sum\limits_{q_1=1}^{\infty}
\sum\limits_{q_2=1}^{\infty}
\frac{\mu\left( q_1 \right)}{\phi\left( q_1 \right)}
\frac{\mu\left( q_2 \right)}{\phi\left( q_2 \right)}
z^{\left( q_1 + q_2 \right) }
\left[
\lim_{N \to \infty}
\frac{1}{N}
\sum\limits_{n = 1}^{N}
e^{2 \pi i \left( \frac{b n + \ell}{a} \right) j }
c_{q_1} \left( n \right)
c_{q_2} \left( \frac{b n + \ell}{a} \right)
\right]
\end{aligned}
\end{equation}
Using the definition of $c_q(x)$ and the fact that $c_q(x)$ is an even function,
the expression in square brackets of \eqref{eq:Theorem9:InterchangedLimits}
can be expanded to:
\begin{equation} \label{eq:ConjectureD:GadiyarPadma:3.5}
\lim_{N \to \infty}
\frac{1}{N}
\sum\limits_{n = 1}^{N}
e^{ 2 \pi i \left( \frac{b n + \ell }{a} \right)  j }
\sum\limits_{\substack{k_1 = 1 \\ (k_1, q_1) = 1}}^{q_1}
e^{-2 \pi i \frac{k_1}{q_1} n}
\sum\limits_{\substack{k_2 = 1 \\ (k_2, q_2) = 1}}^{q_2}
e^{ 2 \pi i \frac{k_2}{q_2} \left( \frac{b n + \ell}{a}  \right) }
\end{equation}
Using \eqref{eq:ConjectureD:GadiyarPadma:3.5} and separating out for $n$, \cref{eq:ConjectureD:StartingPoint} transforms into:
\begin{equation} \label{eq:ConjectureD:GadiyarPadma:3.6}
\begin{aligned}
&\lim_{N \to \infty}
\frac{1}{N}
\sum\limits_{n = 1}^{N}
\left(
\sum\limits_{q_1=1}^{\infty}
\frac{\mu\left( q_1 \right)}{\phi\left( q_1 \right)}
c_{q_1} \left( n \right)
z^{q_1}
\right)
\left(
\sum\limits_{q_2=1}^{\infty}
\frac{\mu\left( q_2 \right)}{\phi\left( q_2 \right)}
c_{q_2} \left( \frac{b n + \ell}{a} \right)
z^{q_2}
\right)
\left(
\frac{1}{a}
\sum\limits_{j = 0}^{a-1} e^{2 \pi i \left( \frac{b n + \ell }{a} \right)  j }
\right)
\\
&=
\frac{1}{a}
\sum\limits_{j = 0}^{a-1}
\sum\limits_{q_1=1}^{\infty}
\sum\limits_{q_2=1}^{\infty}
\sum\limits_{\substack{k_1 = 1 \\ (k_1, q_1) = 1}}^{q_1}
\sum\limits_{\substack{k_2 = 1 \\ (k_2, q_2) = 1}}^{q_2}
\frac{\mu\left( q_1 \right)}{\phi\left( q_1 \right)}
\frac{\mu\left( q_2 \right)}{\phi\left( q_2 \right)}
e^{2 \pi i \left( \frac{k_2}{q_2} + j \right) \frac{\ell }{a} }
z^{\left( q_1 + q_2 \right) }
\left[
\lim_{N \to \infty}
\frac{1}{N}
\sum\limits_{n = 1}^{N}
e^{2 \pi i
\left(
\frac{b}{a} \left( \frac{k_2}{q_2} + j \right)
- \frac{k_1}{q_1}
\right)  n
}
\right]
\end{aligned}
\end{equation}

Gadiyar and Padma proved in
\cite{bibitem:GadiyarPadma:SophieGermainPrimes} and
\cite{bibitem:GadiyarPadma:HardyLittleWoodConjecture:D},
that:
\begin{equation*}
\lim_{N \to \infty}
\frac{1}{N}
\sum\limits_{n = 1}^{N}
e^{2 \pi i
\left(
\left( \frac{k_2}{q_2} + j \right)\left( \frac{b}{a} \right)
- \frac{k_1}{q_1}
\right)  n
}
=
\begin{cases}
1 & \frac{k_1}{q_1} = \left( \frac{k_2}{q_2} + j \right)\left( \frac{b}{a} \right)
\\
0 & \text{otherwise}
\end{cases}
\end{equation*}
With this,
\eqref{eq:ConjectureD:GadiyarPadma:3.6} becomes:
\begin{equation} \label{eq:ConjectureD:GadiyarPadma:3.7a}
\begin{aligned}
&\lim_{N \to \infty}
\frac{1}{N}
\sum\limits_{n = 1}^{N}
\left(
\sum\limits_{q_1=1}^{\infty}
\frac{\mu\left( q_1 \right)}{\phi\left( q_1 \right)}
c_{q_1} \left( n \right)
z^{q_1}
\right)
\left(
\sum\limits_{q_2=1}^{\infty}
\frac{\mu\left( q_2 \right)}{\phi\left( q_2 \right)}
c_{q_2} \left( \frac{b n + \ell}{a} \right)
z^{q_2}
\right)
\left(
\frac{1}{a}
\sum\limits_{j = 0}^{a-1}
e^{2 \pi i \left( \frac{b n + \ell }{a} \right)  j }
\right)
\\
&=
\frac{1}{a}
\sum\limits_{j = 0}^{a-1}
\sum\limits_{q_1=1}^{\infty}
\sum\limits_{q_2=1}^{\infty}
\sum\limits_{\substack{k_1 = 1 \\ (k_1, q_1) = 1}}^{q_1}
\sum\limits_{
\substack{
			 k_2 = 1
		\\ (k_2, q_2) = 1
		\\ \frac{k_1}{q_1} = \left( \frac{k_2}{q_2} + j \right) \frac{b}{a}
	}
}^{q_2}
\frac{\mu\left( q_1 \right)}{\phi\left( q_1 \right)}
\frac{\mu\left( q_2 \right)}{\phi\left( q_2 \right)}
e^{2 \pi i \left( \frac{k_2}{q_2} + j \right) \frac{\ell }{a} }
z^{\left( q_1 + q_2 \right) }
\end{aligned}
\end{equation}
Taking the limit, $z \to 1^-$, yields:
\begin{equation} \label{eq:ConjectureD:GadiyarPadma:3.7}
\begin{aligned}
&\lim_{N \to \infty}
\frac{1}{N}
\sum\limits_{\substack{n = 1 \\ a \mid \left( b n + \ell \right)}}^{N}
\Lambda_1\left( n \right)
\Lambda_1\left( \frac{b n + \ell}{a} \right)
&=
\frac{1}{a}
\sum\limits_{j = 0}^{a-1}
\sum\limits_{q_1=1}^{\infty}
\sum\limits_{q_2=1}^{\infty}
\sum\limits_{\substack{k_1 = 1 \\ (k_1, q_1) = 1}}^{q_1}
\sum\limits_{
\substack{
			 k_2 = 1
		\\ (k_2, q_2) = 1
		\\ \frac{k_1}{q_1} = \left( \frac{k_2}{q_2} + j \right) \frac{b}{a}
	}
}^{q_2}
\frac{\mu\left( q_1 \right)}{\phi\left( q_1 \right)}
\frac{\mu\left( q_2 \right)}{\phi\left( q_2 \right)}
e^{2 \pi i \left( \frac{k_2}{q_2} + j \right) \frac{\ell }{a} }
\end{aligned}
\end{equation}
Gadiyar and Padma proved in
\cite{bibitem:GadiyarPadma:SophieGermainPrimes} and
\cite{bibitem:GadiyarPadma:HardyLittleWoodConjecture:D},
that the right hand side of \eqref{eq:ConjectureD:GadiyarPadma:3.7} is:
\begin{equation*}
\frac{1}{a}
\sum\limits_{j = 0}^{a-1}
\sum\limits_{q_1=1}^{\infty}
\sum\limits_{q_2=1}^{\infty}
\sum\limits_{\substack{k_1 = 1 \\ (k_1, q_1) = 1}}^{q_1}
\sum\limits_{
\substack{
			 k_2 = 1
		\\ (k_2, q_2) = 1
		\\ \frac{k_1}{q_1} = \left( \frac{k_2}{q_2} + j \right) \frac{b}{a}
	}
}^{q_2}
\frac{\mu\left( q_1 \right)}{\phi\left( q_1 \right)}
\frac{\mu\left( q_2 \right)}{\phi\left( q_2 \right)}
e^{2 \pi i \left( \frac{k_2}{q_2} + j \right) \frac{\ell }{a} }
=
\frac{2 C_2}{a}
\prod\limits_{p} \left( \frac{p-1}{p-2}\right)
\end{equation*}
where the product extends over all odd primes $p$ which divide $a$, $b$ or $\ell$.

Since:
\begin{equation} \label{eq:ConjectureD:GadiyarPadma:3.9}
\begin{aligned}
0 &< \frac{2 C_2}{a}
\prod\limits_{p} \left( \frac{p-1}{p-2}\right)
\\
&=
\lim_{N \to \infty}
\frac{1}{N}
\sum\limits_{\substack{n = 1 \\ a \mid \left( b n + \ell \right)}}^{N}
\Lambda_1\left( n \right)
\Lambda_1\left( \frac{b n + \ell}{a} \right)
\\
&=
\lim_{N \to \infty}
\frac{1}{N}
\sum\limits_{\substack{n = 1 \\ a \mid \left( b n + \ell \right)}}^{N}
\left[
\frac{\phi\left( n \right)}{n}
\Lambda\left( n \right)
\right]
\left[
\frac{\phi\left( \frac{b n + \ell}{a}  \right)}{\frac{b n + \ell}{a} }
\Lambda\left( \frac{b n + \ell}{a} \right)
\right]
\\
&\le
\lim_{N \to \infty}
\frac{1}{N}
\sum\limits_{\substack{n = 1 \\ a \mid \left( b n + \ell \right)}}^{N}
\Lambda\left( n \right)
\Lambda\left( \frac{b n + \ell}{a} \right)
\end{aligned}
\end{equation}
Thus, \cref{eq:ConjectureD:GadiyarPadma:3.9} simplifies to:
\begin{equation} \label{eq:ConjectureD:GadiyarPadma:3.10}
0 < 
\frac{2 C_2}{a}
\prod\limits_{p} \left( \frac{p-1}{p-2}\right)
\le
\lim_{N \to \infty}
\frac{1}{N}
\sum\limits_{\substack{n = 1 \\ a \mid \left( b n + \ell \right)}}^{N}
\Lambda\left( n \right)
\Lambda\left( \frac{b n + \ell}{a} \right)
\end{equation}
Since the mean value found in \eqref{eq:ConjectureD:GadiyarPadma:3.10} is greater than zero there 
must be an infinite number of prime pairs which are solutions to $a p - b p^{\prime} = \ell$
\end{proof}

\newpage
\section{Proof of Prime Number Theorem}
Since new tools should be able to solve old problems, in this section the
Ramanujan-Abel method is applied to the prime number theorem.
The result:
\begin{equation*}
\lim_{x \to \infty}
\frac{1}{x}
\sum\limits_{n \le x}
\Lambda_1\left( n \right)
= 1
\end{equation*}
is equivalent to the prime number theorem.

\begin{proof}
Without loss of generality $x$ can be limited to integers.
Starting with:
\begin{equation} \label{eq:PNT:RamanujanAbel:Start}
\lim_{N \to \infty}
\frac{1}{N}
\sum\limits_{n = 1}^{N}
\sum\limits_{q = 1}^{\infty}
\frac{\mu\left( q \right)}{\phi\left( q \right)}
c_q\left( n \right)
z^q
=
\sum\limits_{q = 1}^{\infty}
\frac{\mu\left( q \right)}{\phi\left( q \right)}
z^q
\left[
\lim_{N \to \infty}
\frac{1}{N}
\sum\limits_{n = 1}^{N}
c_q\left( n \right)
\right]
\end{equation}
and justifying the interchange of limits because of the uniform convergence
established by \cref{thm:UniformConvergence:RFPowerSeries:Lambda1}.
From \Cref{appendix:Properties:RamanujanSums}:
\begin{equation*}
\lim_{N \to \infty}
\frac{1}{N}
\sum\limits_{n = 1}^{N}
c_q\left( n \right)
=
\begin{cases}
1 & q = 1
\\
0 & \text{otherwise}
\end{cases}
\end{equation*}
so \eqref{eq:PNT:RamanujanAbel:Start} becomes,
\begin{equation} \label{eq:PNT:RamanujanAbel:Penultimate}
\begin{aligned}
\lim_{N \to \infty}
\frac{1}{N}
\sum\limits_{n = 1}^{N}
\sum\limits_{q = 1}^{\infty}
\frac{\mu\left( q \right)}{\phi\left( q \right)}
c_q\left( n \right)
z^q
&=
\sum\limits_{q = 1}^{1}
\frac{\mu\left( q \right)}{\phi\left( q \right)}
c_q\left( n \right)
z^q
\\
&=
\frac{\mu\left( 1 \right)}{\phi\left( 1 \right)}
z^1
\\
&=
z
\end{aligned}
\end{equation}
Applying the limit, $z \to 1^-$, to \eqref{eq:PNT:RamanujanAbel:Penultimate} yields:
\begin{equation*}
\lim_{N \to \infty}
\frac{1}{N}
\sum\limits_{n = 1}^{N}
\Lambda_1\left( n \right)
= 1
\end{equation*}

\end{proof}

\newpage
\section{The Parity Problem}
The deep connection between Ramanujan sums, $c_q\left( n \right)$,
and the circle method are discussed in
\cite{bibitem:GadiyarPadma:LinkingCircleAndSieve} and
\cite{bibitem:GadiyarPadma:RotaMeetsRamanujan}.
Using Ramanujan-Fourier expansions of arithmetic functions within a sieve
places the Ramanujan sums with a foot in both worlds.
As noted by Hardy, Ramanujan sums
\textit{
have a peculiar interest because they show explicitly the source of the irregularities in the behaviour of their sums.
Thus, for example
\begin{equation*}
\sigma\left( n \right) = \frac{\pi^2 n}{6}
\left(
1
+ \frac{\left( -1 \right)^n }{2^2}
+ \frac{2 \cos\left( \frac{2}{3} \right) n \pi}{3^2}
+ \frac{2 \cos\left( \frac{1}{2} \right) n \pi }{4^2}
+ \cdots
\right)
\end{equation*}
and we see at once that the most important term in $\sigma(n)$ is $\frac{1}{6} \pi^2 n$, and that irregular
variations about this average value are produced by a series of harmonic oscillations of
decreasing amplitude.
}

As pointed out in \cite{bibitem:GadiyarPadma:WienerKhinchin}: \textit{
The Ramanujan-Fourier series trap the vagaries in the behaviour
of primes and the Wiener-Khintchine formula traps their correlation properties
}

The Wiener-Khintchine formula is specifically designed to study signals and other processes with chaotic features resembling Brownian motion.
The distribution of primes and twin primes seem to have this kind of chaotic structure,
so, it is reasonable to expect the Weiner-Khintchine approach will avoid the brunt of  the parity problem because it takes
advantage of the smoothing introduced by the convolution process (i.e. the mean-value process).
Weiner-Khintchine  is specifically designed to provide information on the convolution of signals which are
too ``rough'' individually have a Fourier transform of their own.
Sieving methods tend to accumulate local errors and, generally, do not take advantage of any the global smoothing or
widely separated cancellations which may be introduced by the convolution (mean-value) process.
The Wiener-Khintchine theorem does take advantage of such global  cancellations.

Because of the significance of the rational points of the unit circle within both the circle method and the Ramanujan sums,
Ramanujan-Fourier expansions capture the probabilistic elements of a sieve while retaining the rigor of the circle method.

\newpage
\section{Further Directions}
The Ramanujan-Abel method provides a robust framework in which to investigate
the distribution of primes and prime m-tuples of a given algebraic form.

\subsection{Proof of the m-tuple Conjecture}
In \cite{bibitem:GadiyarPadma:HardyLittleWoodConjecture:D} Gadiyar and Padma provide the heuristic argument for the result:
\begin{theorem}[Conjecture $X_1$ of of Hardy and Littlewood]
Let
$a_1$, $a_2$, $\cdots$, $a_{m}$,
be distinct integers such that,
$n, \left( n + a_1 \right)$, $\left( n + a_2 \right)$, $\cdots$, $\left( n + a_{m} \right)$,
meet the criteria of \Cref{def:BunyakovskyCondition:ChebyshevLikeFunction}
Then:
\begin{equation*}
\lim_{N \to \infty}
\frac{1}{N}
\sum\limits_{n=1}^{N}
\Lambda\left( n \right)
\Lambda\left( n + a_1 \right)
\Lambda\left( n + a_2 \right)
\cdots
\Lambda\left( n + a_m \right)
=
\prod\limits_{p}
\left( \frac{p}{p - 1} \right)^{m}
\frac{p-\nu}{p - 1}
\end{equation*}
where $\nu = \nu_{r} = \nu\left( p;\text{ } 0, a_1, a_2, \cdots, a_{m}\right)$
is the number of distinct residues of
$a_1$, $a_2$, $\cdots$, $a_{m}$ to modulus $p$
\end{theorem}

Using the Ramanujan-Abel method
proves this heuristic result is correct.
The use of the Ramanujan expansion power series and the uniform convergence thereof
provides for the interchange of limits needed to
reach the result found in \cite{bibitem:GadiyarPadma:HardyLittleWoodConjecture:D}

\subsection{Nearly-Square Primes and Other Conjectures}
There are several more conjectures in \cite{bibitem:HardyLittlewood01}
which deal with the density of primes of various algebraic forms
and with the density of prime pairs and prime $m$-tuplets of various algebraic forms.
These conjectures include:

\begin{table}[tbhp]
\caption{Hardy-Littlewood Conjectures}
\label{tbl:HardyLittlewoodConjectures}
\begin{tabular}{ | l  l  l |}
\hline
Conjecture E & primes of the form $m^2 + 1$ are infinite  & page 48\\
Conjecture F & primes of the form $a m^2 + b m + c$ are infinite  & page 48\\
Conjecture G & prime pairs of the form $p, p^{\prime} = a m^2 + b m + p$ are infinite  & page 49\\
Conjecture J & prime pairs of the form $p, p^{\prime} = m_1^2 + m_2^2 + p$ are infinite  & page 50\\
Conjecture J & prime pairs of the form $p, p^{\prime} = m_1^2 + m_2^2 + m_3^2 + m_4^2 + p$ are infinite  & page 50\\
Conjecture K & primes of the form $m^3 + k$ are infinite  & page 50 \\
Conjecture M & primes of the form $\ell^3 + m^3 + k$ are infinite  & pages 51-2\\
Conjecture N & primes of the form $\ell^3 + m^3 + k^3$ are infinite  & page 52 \\
Conjecture P & prime pairs of the form $m^2 + 1, m^2 + 3$ are infinite  & page 62 \\
Conjecture $X_1$ & prime $m$-tuplets of the form $n + b_1, n + b_2, \cdots n + b_m, $ are infinite  & pages 52-61 \\
\hline
\begin{tabular}{@{}l@{}}
Bateman-Horn \\
Conjecture
\end{tabular}
&
\begin{tabular}{@{}l@{}}
For any Bunyakovsky set, there exists
\\ an infinite number of solutions, $n$, such that
\\ $f_1(n), f_2(n), \cdots, f_m(n),$ are all prime.
\end{tabular}
&
\begin{tabular}{@{}l@{}}
   \cite{bibitem:Baier:Bateman-Horn}, \cite{bibitem:Bateman-Horn01},
\\ and \cite{bibitem:Conrad:HardyLittlewoodConstants}
\end{tabular}
\\
\hline
\end{tabular}
\end{table}

The author believes the Ramanujan-Able method will prove useful in confirming the conjectures
of \Cref{tbl:HardyLittlewoodConjectures},
but, it requires a better understanding of expressions found in \Cref{tbl:RamunjanSums:HardyLittlewoodConjectures}.

\begin{table}[tbhp]
\caption{Ramanujan Sums for the Various Conjectures}
\label{tbl:RamunjanSums:HardyLittlewoodConjectures}
\begin{tabular}{ | l  l |}  
\hline
Conjecture E
&
$ 
\lim\limits_{N \to \infty}
\frac{1}{N}
\sum\limits_{n = 1}^{N}
c_q\left( n^2 + 1 \right)
$ 
\\
Conjecture F
&
$ 
\lim\limits_{N \to \infty}
\frac{1}{N}
\sum\limits_{n = 1}^{N}
c_q\left( a n^2 + b n + c \right)
$ 
\\
Conjecture G
&
$ 
\lim\limits_{N \to \infty}
\frac{1}{N}
\sum\limits_{n = 1}^{N}
c_{q_1}\left( n \right)
c_{q_2}\left( a m^2 + b m + n \right)
$ 
\\
Conjecture J
&
$ 
\lim\limits_{N \to \infty}
\frac{1}{N}
\sum\limits_{n = 1}^{N}
c_{q_1}\left( n \right)
c_{q_2}\left( m_1^2 + m_2^2 + n \right)
$ 
\\
Conjecture J
&
$ 
\lim\limits_{N \to \infty}
\frac{1}{N}
\sum\limits_{n = 1}^{N}
c_{q_1}\left( n \right)
c_{q_2}\left( m_1^2 + m_2^2 + m_3^2 + m_4^2 + n \right)
$ 
\\
Conjecture K
&
$ 
\lim\limits_{N \to \infty}
\frac{1}{N}
\sum\limits_{n = 1}^{N}
c_q\left( n^3 + k \right)
$ 
\\
Conjecture M
&
$ 
\lim\limits_{N \to \infty}
\frac{1}{N}
\sum\limits_{n = 1}^{N}
c_q\left( n^3 + m^3 + k \right)
$ 
\\
Conjecture N
&
$ 
\lim\limits_{N \to \infty}
\frac{1}{N}
\sum\limits_{n = 1}^{N}
c_q\left( n^3 + m^3 + k^3 \right)
$ 
\\
Conjecture P
&
$ 
\lim\limits_{N \to \infty}
\frac{1}{N}
\sum\limits_{n = 1}^{N}
c_{q_1}\left( n^3 + 1 \right)
c_{q_2}\left( n^3 + 3 \right)
$ 
\\
Conjecture $X_1$
&
$ 
\lim\limits_{N \to \infty}
\frac{1}{N}
\sum\limits_{n = 1}^{N}
c_{q_1}\left( n + b_1 \right)
c_{q_2}\left( n + b_2 \right)
\cdots
c_{q_m}\left( n + b_m \right)
$ 
\\
\begin{tabular}{@{}l@{}}
Bateman-Horn \\
Conjecture
\end{tabular}
&
$ 
\lim\limits_{N \to \infty}
\frac{1}{N}
\sum\limits_{n = 1}^{N}
c_{q_1}\left( f_1\left( n \right) \right)
c_{q_2}\left( f_2\left( n \right) \right)
\cdots
c_{q_m}\left( f_m\left( n \right) \right)
$ 
\\
\hline
\end{tabular}
\end{table}

With the function, $E_g\left(m_1, \cdots ,m_r)\right)$,
T\'oth \cite{bibitem:Toth:01} begins to explore these expressions,
but much more research needs to be done in this area.

\subsection{Other Ramanujan-Fourier Expansions}
Currently, the most commonly cited Ramanujan-Fourier expansions
for arithmetic functions are:
\begin{align}
d\left( n \right) &=
-\sum\limits_{q=1}^{\infty}
\frac{\log{q}}{q} c_q\left( n \right)
\\
\sigma\left( n \right) &=
\frac{\pi^2 n}{6}
\sum\limits_{q=1}^{\infty}
\frac{c_q(n)}{q^2}
\\
\Lambda_1\left( n \right) &=
\sum\limits_{q=1}^{\infty}
\frac{\mu(q)}{\phi(q)}
c_q(n)
\\
r\left( a \right) &=
\pi
\sum\limits_{q=1}^{\infty}
\frac{
\left( -1 \right)^{q - 1}
}{
2 q - 1
}
c_{2 q - 1}\left( a \right)
\end{align}
Where:
\begin{itemize}
\item $d\left( n \right)$ is the divisor function,
\item $\sigma\left( n \right)$ is the sum of divisors function,
\item $\Lambda_1\left( n \right)$ is the modified von Mangodlt function
\item $r\left( a \right)$ is the number of integer lattice points, $\left( u, v \right)$, within the circle: $u^2 + v^2 \le a$.
\end{itemize}

Are there other arithmetic functions used in number theory for which it
would useful to have a Ramanujan-Fourier expansion?
For example, does the prime indicator function:
\begin{align*}
\varepsilon\left( n \right)
&=
\frac{
\mu\left( n \right)
\Lambda\left( n \right)
}{
\log{n}
}
\\
&=
\begin{cases}
-1 & n \text{ is prime}
\\
0 & n \text{ is composite}
\end{cases}
\end{align*}
have a Ramanujan-Fourier expansion of the form?
\begin{equation*}
\varepsilon\left( n \right)
=
\sum\limits_{q=1}^{\infty}
{\widehat{\varepsilon}}_q
c_q\left( n \right)
\end{equation*}

\subsection{Other Questions in Additive Number Theory}
The mean-value properties listed in \Cref{appendix:Properties:RamanujanSums}
take advantage of the fact the $N$ tends to $\infty$.  There are
several conjectures within \cite{bibitem:HardyLittlewood01} where the sieve under
consideration does not range over all natural numbers.
For example, Conjecture A is the strong Goldbach conjecture,
where the Ramanujan-Fourier expression under consideration is:
\begin{equation} \label{eq:StrongGoldbachSieve}
\begin{aligned}
\sum\limits_{n = 1}^{2 N}
&\Lambda_1\left( n \right)
\Lambda_1\left(  N - n \right)
\\
&=
\lim\limits_{r \to 1^-}
\sum\limits_{n = 1}^{2 N}
\sum\limits_{q_1=1}^{\infty}
\frac{\mu(q_1)}{\phi(q_1)}
c_{q_1}(n)
r^{q_1}
\sum\limits_{q_2=1}^{\infty}
\frac{\mu(q_2)}{\phi(q_2)}
c_{q_2}(2 N - n)
r^{q_2}
\end{aligned}
\end{equation}
Even with an interchange of limits transforming  the right hand side of
\eqref{eq:StrongGoldbachSieve} to:
\begin{equation*}
\lim\limits_{r \to 1^-}
\sum\limits_{q_1=1}^{\infty}
\sum\limits_{q_2=1}^{\infty}
\frac{\mu(q_1)}{\phi(q_1)}
\frac{\mu(q_2)}{\phi(q_2)}
r^{\left( q_1 + q_2 \right)}
\left[
\sum\limits_{n = 1}^{2 N}
c_{q_1}(n)
c_{q_2}(2 N - n)
\right]
\end{equation*}
the value of and bounds on:
\begin{equation*}
\sum\limits_{n = 1}^{2 N}
c_{q_1}(n)
c_{q_2}(2 N - n)
\end{equation*}
are not well understood at this time.

\newpage
\section{Conclusions}
The Ramanujan-Abel method is a robust framework to investigate questions regarding the
distribution of prime pairs and other prime-tuples.
The method is sufficient to prove Conjectures B and D of Hardy and Littlewood.
This is a significant result, but it is clear there there is more investigation to be done.

The author is grateful for the opportunity to effect the interchange of limits needed within the works of
Gadiyar and Padma. Given that the author received his initial  training in mathematical analysis
at the Milwaukee School of Engineering
the conclusion of Gadiyar and Padma from  \cite{bibitem:GadiyarPadma:WienerKhinchin} is best:
\\

\textit{
It is a pleasant surprise that the Wiener-Khintchine formula which normally occurs in
practical problems of Brownian motion, electrical engineering and other applied areas
of technology and statistical physics has a role in the behaviour of prime numbers
which are studied by pure mathematicians
}

\appendix

\newpage
\section{Some Properties of Ramanujan Sums} \label{appendix:Properties:RamanujanSums}
The Ramanujan sum, $c_q(n)$, is defined as
\begin{equation*}
c_q(n) = \sum\limits^{q}_{\substack{k=1 \\ (k,q)=1}} e^{2 \pi i \frac{k}{q} n }
\end{equation*}

Some of the properties of $c_q(n)$ and the mean value of $c_q(n)$ are:
\begin{enumerate}[ a) ] \label{List:Properties:RamanujanSums}
\item
\begin{equation*}
c_1(n) = 1
\end{equation*}
\item
\begin{equation*}
c_q(0) = \phi(q)
\end{equation*}
\item
\begin{equation*}
c_q(1) = \mu(q)
\end{equation*}
\item
\begin{equation*}
c_q(n) = \begin{cases}
\phi(q) & q \mid  n
\\
-1 & q \nmid  n
\end{cases}
\end{equation*}
\item
\begin{equation*}
c_{rs}(n) = c_{s}(n) c_{s}(n) \text{ where } \left( r, s\right) = 1
\end{equation*}
\item
\begin{equation*}
\left\vert c_q(n) \right\vert
\le
\phi\left( q \right)
\end{equation*}
\item
\begin{equation*}
\left\vert c_q(n) \right\vert
\le
\sigma\left( n \right)
\end{equation*}
\item
\begin{equation*}
c_q(n) = c_q(-n)
\end{equation*}
\item
\begin{equation*}
c_q(n) = c_{-q}(n)
\end{equation*}
\item
\begin{equation*}
\lim_{N \to \infty}
\frac{1}{N}
\sum_{n=1}^{N}
c_q(n) =
\begin{cases}
1 & q=1 \\
0 & \text{otherwise} \\
\end{cases}
\end{equation*}
\item
\begin{equation*}
\lim_{N \to \infty}
\frac{1}{N}
\sum_{n=1}^{N}
c_r(n)
c_s(n)
=
\begin{cases}
\phi\left( r \right) & r=s \\
0 & \text{otherwise} \\
\end{cases}
\end{equation*}
\item
\begin{equation*}
\lim_{N \to \infty}
\frac{1}{N}
\sum_{n=1}^{N}
c_r(n) c_s(n \pm m) =
\begin{cases}
c_r(m) & r = s \\
0 &\text{otherwise}
\end{cases}
\end{equation*}
\end{enumerate}
Where:
\begin{itemize}
\item $q$ is a positive integer
\item $n \in \mathbb{Z}$
\end{itemize}

\newpage
The Ramanujan sums can be extended to real values by accepting the definition:
\begin{equation*}
c_q(x) =
\begin{cases}
1 & q = 0
\\
\cos\left( 2 \pi x \right) & q = 1
\\
\cos\left( \pi x \right) & q = 2
\\
2 \sum\limits_{\substack{k=1 \\ (k,q)=1}}^{\left\lfloor \frac{q}{2} \right\rfloor}
\cos\left( 2 \pi \frac{k}{q} x \right)
& q \ge 3
\end{cases}
\end{equation*}
The real-valued Ramanujan sums maintain many of the properties of $c_q(n)$.
Some of the properties of $c_q(x)$ and the mean value of $c_q(x)$ are:
\begin{enumerate}[ a) ]
\item
\begin{equation*}
c_q(x) = c_q(n) \text{ } x \in \mathbb{Z} \text{ and } 0 < q
\end{equation*}
\item
\begin{equation*}
c_q(0) = \phi(q)
\end{equation*}
\item
\begin{equation*}
c_q(1) = \mu(q)
\end{equation*}
\item
\begin{equation*}
c_{rs}(x) = c_{s}(x) c_{s}(x) \text{ where } \left( r, s\right) = 1
\end{equation*}
\item
\begin{equation*}
\left\vert c_q(x) \right\vert
\le
\phi\left( q \right)
\end{equation*}
\item
\begin{equation*}
\left\vert c_q(x) \right\vert
\le
\sigma\left( x \right)
\end{equation*}
\item
\begin{equation*}
c_q(x) = c_q(-x)
\end{equation*}
\item
\begin{equation*}
c_{q}(x) = c_{-q}(x)
\end{equation*}
\item
\item
\begin{equation*}
\lim_{T \to \infty}
\frac{1}{2 T}
\int\limits_{x=-T}^{T}
c_q(x) dx
=
\begin{cases}
1 & q=0 \\
0 & \text{otherwise} \\
\end{cases}
\end{equation*}
\item
\begin{equation*}
\lim_{T \to \infty}
\frac{1}{2 T}
\int\limits_{x=-T}^{T}
c_r(x)
c_s(x)
dx
=
\begin{cases}
\phi\left( r \right) & r=s \\
0 & \text{otherwise} \\
\end{cases}
\end{equation*}
\item
\begin{equation*}
\lim_{T \to \infty}
\frac{1}{2 T}
\int\limits_{x=-T}^{T}
c_{r}(x) c_{s}(x \pm y) dx
=
\begin{cases}
c_{r}(y) & r = s \\
0 & \text{otherwise}
\end{cases}
\end{equation*}
\end{enumerate}
Where:
\begin{itemize}
\item $q$ is a non-negative integer
\item $n \in \mathbb{Z}$
\item $x \in \mathbb{R}$
\item The value of $\phi\left( 0 \right)$ is understood to be: $\phi\left( 0 \right) = 1$:
\item The value of $\mu\left( 0 \right)$ is understood to be: $\mu\left( 0 \right) = 1$:
\end{itemize}


\newpage
\begin{thebibliography}{10}

\bibitem{bibitem:Baier:Bateman-Horn}
S.Baier,
"On the Bateman-Horn conjecture",
\textit{J. Number Theory}, vol. 96, pp 432-448, 2002

\bibitem{bibitem:Bateman-Horn01}
P. T. Bateman, and R. A. Horn,
A heuristic asymptotic formula concerning the distribution of prime numbers,
Math. Comput. 16 (1962), 363–367.

\bibitem{bibitem:Besicovitch:01}
A. S. Besicovitch,
Almost Periodic Functions,
Dover Publications, Inc. (1954)

\bibitem{bibitem:Carmichael01}
R. D. Carmichael,
Expansions of arithmetical functions in infinite series,
\textit{Proc. London Math. Soc.} (2) 34 (1932), 1-26.

\bibitem{bibitem:Conrad:HardyLittlewoodConstants}
Keith Conrad,
Hardy-Littlewood Constants
\href{http://www.math.uconn.edu/~kconrad/articles/hlconst.pdf}{http://www.math.uconn.edu/~kconrad/articles/hlconst.pdf}

\bibitem{bibitem:GadiyarPadma:WienerKhinchin}
H. G. Gadiyar and R. Padma,
Ramanujuan-Fourier series, the Wiener-Khinchin formula and the distribution of prime pairs,
\href{http://www.secamlocal.ex.ac.uk/people/staff/mrwatkin/zeta/padma.pdf}{Physica A 269(1999) 503-510.}

\bibitem{bibitem:GadiyarPadma:SophieGermainPrimes}
\bysame,
Ramanujan - Fourier Series and the Density of Sophie Germain Primes,
\href{http://arxiv.org/abs/math/0508639}{arXiv:math/0508639v1}

\bibitem{bibitem:GadiyarPadma:LinkingCircleAndSieve}
\bysame,
\textit{Linking the Circle and the Sieve: Ramanujan - Fourier Series},
\href{http://arxiv.org/abs/math/0601574}{arXiv:math/0601574v1}

\bibitem{bibitem:GadiyarPadma:RotaMeetsRamanujan}
\bysame,
Rota meets Ramanujan: Probabilistic interpretation of Ramanujan - Fourier series
\href{http://arxiv.org/abs/math-ph/0209066}{arXiv:math-ph/0209066v1}

\bibitem{bibitem:GadiyarPadma:HardyLittleWoodConjecture:D}
\bysame,
Ramanujan-Fourier Series and The Conjecture D of Hardy and Littlewood
\textit{Czechoslovak Mathematical Journal}, 64 (139) (2014), 251–267


\bibitem{bibitem:Hardy01}
G. H. Hardy,
Note on Ramanujan's trigonometrical function $c_q(n)$ and certain series of arithmetical functions,
Proc. Camb. Phil. Soc. 20, (1921) 263 - 271.

\bibitem{bibitem:HardyLittlewood01}
G. H. Hardy and J. E. Littlewood,
Some problems of 'Partitio numerorum'; III - On the expression of a number as a sum of primes,
\textit{Acta Math.},
Vol.44, pp.1–70, 1922.

\bibitem{bibitem:LuVaswani:01}
Wei Lu and Namrata Vaswani,
The Wiener-Khinchin Theorem For Non-Wide Sense Stationary Random Processes
Submitted to IEEE International Conference on Acoustics, Speech, and Signal Processing (ICASSP) 2009
\href{http://arxiv.org/abs/0904.0602}{arXiv:0904.0602} [math.ST]

\bibitem{bibitem:LuVaswani:02}
\bysame,
Appendix to ``The Wiener-Khinchin Theorem For Non-Wide Sense Stationary Random Processe''
Submitted to IEEE International Conference on Acoustics, Speech, and Signal Processing (ICASSP) 2009
\href{http://citeseerx.ist.psu.edu/viewdoc/download?doi=10.1.1.362.3826&rep=rep1&type=pdf}{Appendix A}

\bibitem{bibitem:Lucht:01}
L.G. Lucht,
A survey of Ramanujan expansions,
Int. J. Number Theory 6(8) (2010) 1785–1799.

\bibitem{bibitem:Murty:01}
M. Ram Murty,
Problems in Analytic Number Theory,
Graduate Texts in Mathematics/Readings in Mathematics, vol.206, 2nd edition, Springer, New York, 2008.

\bibitem{bibitem:Murty:02}
\bysame,
Ramanujan series for arithmetical functions,
Hardy-Ramanujan J. 36 (2013) 21–33.

\bibitem{bibitem:ParsevalEquation_by_MPeter}
Manfred Peter,
\href{http://books.google.com/books?id=h4\_9GwAACAAJ}{The Parseval Equation for Almost Periodic Arithmetical Functions} \\
Albert-Ludwigs-Univ., Math. Fak. (2000)

\bibitem{bibitem:Ramanujan01}
S. Ramanujan,
On certain trigonometrical sums and their applications in the theory of numbers,
Trans. Camb. Phil. Soc. 22 (1918), 259 - 276.

\bibitem{bibitem:SchwarzSpilker:01}
W. Schwarz, J. Spilker,
Arithmetical Functions,
London Mathematical Society Lecture Note Series, vol.184, Cambridge University Press, Cambridge, 1994.

\bibitem{bibitem:Sivaramakrishnan:01}
R. Sivaramakrishnan,
Classical Theory of Arithmetic Functions,
Monographs and Textbooks in Pure and Applied Mathematics, vol.126, Marcel Dekker, Inc., New York, 1989.

\bibitem{bibitem:Spilker:01}
J. Spilker,
Ramanujan expansions of bounded arithmetic functions,
Arch. Math. (Basel) 35(5) (1980) 451–453.

\bibitem{bibitem:Toth:01}
L\'aszl\'o T\'oth
Sums of products of Ramanujan sums
Annali Dell'Universita' di Ferrara
May 2012, Volume 58, Issue 1, pp 183-197

\end{thebibliography}
\end{document}